\documentclass[10pt,conference]{ita_style} 
\usepackage[top=1.2in, bottom=1.2in, left=1.2in, right=1.2in]{geometry}
\usepackage{amsfonts, amssymb, amsmath, dsfont, graphicx, subcaption}
\usepackage[linktoc=page, colorlinks, linkcolor=blue, citecolor=blue]{hyperref}
\usepackage{enumerate}
\numberwithin{equation}{section}
\usepackage{fullpage}
\usepackage{authblk}

\DeclareMathOperator{\E}{\mathbb{E}}

\DeclareMathOperator{\argmax}{argmax}

\def \R {\mathbb{R}}

\newtheorem{theorem}{Theorem}[section]
\newtheorem{proposition}[theorem]{Proposition}

\newtheorem{remark}[theorem]{Remark}

\begin{document}
\bibliographystyle{abbrv}

\title{Sketching for Motzkin's Iterative Method for Linear Systems}
\author[*,1]{Elizaveta Rebrova}
\author[*,2]{Deanna Needell}
\affil[*]{Department of Mathematics, University of California, Los Angeles, Los Angeles, California}
\affil[1]{rebrova@math.ucla.edu}
\affil[2]{deanna@math.ucla.edu}



\maketitle
\begin{abstract}
Projection-based iterative methods for solving large over-determined linear systems are well-known for their simplicity and computational efficiency. It is also known that the correct choice of a sketching procedure (i.e., preprocessing steps that reduce the dimension of each iteration) can improve the performance of iterative methods in multiple ways, such as, to speed up the convergence of the method by fighting inner correlations of the system, or to reduce the variance incurred by the presence of noise. In the current work, we show that sketching can also help us to get better theoretical guarantees for the projection-based methods. Specifically, we use good properties of  Gaussian sketching to prove an accelerated convergence rate of the sketched relaxation (also known as Motzkin’s) method. The new estimates hold for linear systems of arbitrary structure. We also provide numerical experiments in support of our theoretical analysis of the sketched relaxation method.
\end{abstract}

\section{Introduction}
We are interested in solving large-scale systems of linear equations 
\begin{equation}\label{main_system}
A x = b,
\end{equation} 
where $b \in \R^m$ and $A$ is an $m \times n$ matrix. We will consider the case when $m \ge n$, when iterative methods are typically employed.

One of the most popular iterative methods is the Kaczmarz method, being simple, efficient and well-adapted to large amounts of data. The original Kaczmarz method starts with some initial guess $x_0$, and then iteratively projects the previous approximation $x_k$ onto the solution space of the next equation in the system. Namely, if $a_1, \ldots, a_m \in \R^n$ are the rows of $A$, then the $k$-th step of the algorithm is
\begin{equation}\label{iter}
x_{k} = x_{k-1}+ \frac{b_i - \langle a_i, x_{k-1}\rangle}{\|a_i\|^2} a_i,
\end{equation}
where $b = (b_1, \ldots, b_n) \in \R^m$ is the right hand side of the system, $x_{k-1} \in \R^n$ is the approximation of a solution $x_*$ obtained in the previous step. 
To provide theoretical guarantees for the convergence of the method, Strohmer and Vershynin \cite{StrVer} proposed to choose the next row index $i$ at random with the probabilities weighted proportionally to the $L_2$ norms of the rows $a_i$. The authors have shown that this \emph{randomized Kaczmarz algorithm} is guaranteed to converge exponentially in expectation, namely,
\begin{equation}\label{standard_convergence_rate}
\E \|x_k - x_*\|^2_2 \le \left(1 - \frac{1}{\tilde\kappa(A)}\right)^k \|x_0 - x_*\|^2_2,
\end{equation}
where $\tilde\kappa(A) = \|A\|^2_F/s_{min}^2(A)$ is a condition number of the system and $x_*$ is the solution of the system \eqref{main_system}.

Here and further, $s_{min}(A)$ denotes the smallest singular value of the matrix $A$ and $\|A\|_F := trace(\sqrt{A^*A})$ (Frobenius, or Hilbert-Shmidt, norm of the matrix). We always assume that the matrix $A$ has full column rank, so that $s_{min}(A) > 0$.

However, the randomized scheme does not always choose the best equation on which to project. The approach by Agmon~\cite{Agm} and Motzkin and Schoenberg~\cite{MotSch}, 
resolves this issue. The idea is to always project onto the farthest equation from the current iteration point $x_k$, making the distance to the solution $\|x_{k+1} - x_*\|_2$ as good as possible by the next step. Precisely, the $k$-th step of the algorithm is defined by \eqref{iter}, where
\begin{equation}\label{condition}
 i := \argmax_{[j]} (A_j x_k - b_j)^2.
\end{equation}
This so-called Motzkin's method has also been known as the
Kaczmarz method with the “most violated constraint” or “maximal-residual” control (\cite{PetPop, NutSepLarSchKoeVir, Cen}).

It is intuitive to expect that Motzkin's method converges more efficiently than Kaczmarz \emph{per iteration} (although each iteration is clearly much more expensive, as one has to swipe over all $m$ equations to determine the best one to satisfy~\eqref{condition}) A step towards the theoretical guarantee of the convergence of Motzkin's method improved convergence rate was made in~\cite{HadNee}, where the authors have shown that 
\begin{equation}\label{accelerated_convergence_rate}
\E \|x_k - x_*\|^2_2 \le \prod_{i = 1}^{k-1}\left[1 - \frac{s_{min}^2(A)}{\gamma_i(A)}\right]^k \|x_0 - x_*\|^2_2,
\end{equation}
where $\gamma_i(A) := \|Ax_i - Ax_*\|_2^2/\|Ax_i - Ax_*\|_{\infty}^2 $ is a so-called \emph{dynamic range} of the system. 

In general, to get a good estimate on the dynamic range, one would need to know more about the structure of the matrix $A$. The authors also give a heuristic for $\gamma_i(S)$ in the case when the matrix $A$ has independent standard normal entries $\gamma_k \sim \|A\|_F^2n/\log(m-k)$ (\cite[Section~2.1]{HadNee}). So, it shows accelerated convergence for several initial iterations while $\log(m-k) > n$.

An independent branch of research improving Kaczmarz method was presented by Gower and Richtarik in \cite{GowRic}. The main idea of their \emph{sketch-and-project} framework is the following. One can observe that a random selection of a row can be represented as a \emph{sketch}, that is, left multiplication by a random vector (with exactly one non-zero entry). Then, the iteration~\eqref{iter} is a projection onto the image of the sketch.  A natural way to extend the method is to pre-process every iteration by left multiplication with an $s \times m$  matrix $S$ (or $1 \times m$ vector) taken from some distribution, thereby reducing the dimension of the problem (see, e.g., \cite{GowRic, loizou2017momentum, needell2019block}). 

\section{Proposed method}
One could also use sketching to pre-process each iteration of Motzkin's method. Like in the original sketch-and-project framework, the idea is to reduce the dimension of the problem (in this case the space of search for the best equation in the sense of~\eqref{condition}), hopefully, preserving the main properties of the system. 

Again, the $k$-th step of the algorithm is defined by \eqref{iter}, but we would choose the best row of the sketched matrix, namely.
\begin{equation}\label{sketched condition}
i = argmax_{[i]} (S^TA)_i x_k - (S^Tb)_i)^2
\end{equation}

\subsection{Block Sketched Motzkin (SKM)}
A natural choice of the sketch is a \emph{block sketch} (analogous to the one considered in \cite{NeeTro} for the Kaczmarz algorithm): ${m \times s}$ sketches 
$$S = (\text{zeros}(s, \text{shift}), I(s, s), \text{zeros}(s, m - \text{shift} - s))^T,$$ 
where zeros() denotes a matrix of all zeros, $I()$ the identity, and $s$ is the block size and $\text{shift}\,=\,sz$, $z \in \{1, 2, ..., \lfloor m/s\rfloor\}$ is selected randomly at each step. 

Then, application of the sketch is exactly equivalent to selecting an $s\times n$ row submatrix of $A$ and proceeding with the Motzkin's scheme inside the selected submatrix. This method is equivalent to the SKM method studied in the more general case of inequalities in \cite{DelHadNee}. To the best of our knowledge, there are still no theoretical guarantees for the convergence of SKM (at least, unspecified to some model of the matrix $A$, such as a Gaussian model).

\subsection{Gaussian Sketched Motzkin (GSM)}

We propose to draw sketch matrices $S$ from the Gaussian distribution on each step ($m \times s$ matrices with i.i.d. $N(0,1)$ entries independent between each other). In this framework, we can provide theoretical guarantees for the exponential convergence rate, $log$-accelerated with respect to the standard rate ~\eqref{standard_convergence_rate}, without any assumptions about the model of the matrix $A$.
\begin{theorem}\label{gauss_motzkin}
Suppose $A$ is a $m \times n$ matrix with full column rank ($m \ge n$) and let $x_*$ be a solution of the system $Ax = b$. For any initial estimate $x_0$, the GSM method produces a sequence $\{x_k, k \ge 0\}$ of iterates that satisfy
\begin{align}\label{c_rate_motzkin}
\E&\|x_{k} - x_*\|_2^2 \le \left( 1 - c \frac{\log s}{\tilde\kappa(A)}\right)^k \|x_0- x_*\|_2^2.
\end{align}
Here, $c > 0$ is an absolute constant.
\end{theorem}

\subsection{Sparse Gaussian sketched Motzkin (sGSM)}
Comparing to SKM, the main disadvantage of GSM is the cost of each sketch: instead of drawing a sub-block of the original matrix $A$, one needs to pre-multiply $A$ by a dense Gaussian matrix. As a trade-off between these two methods, we also propose the \emph{sparse GSM} (or, sGSM) method, taking sketches
\begin{equation}\label{sparse_sample}
S := (\text{zeros}(s, \text{shift}), X, \text{zeros}(s, m - \text{shift} - s))^T,
\end{equation}
where $X$ is an $s \times s$ matrix with i.i.d. $N(0,1)$ entries. At each step, we generate a sketch matrix $S$ independently from the previous steps. This method is equivalent to selecting a $s \times n$ row sub-block of $A$ and then sketching only that sub-block by a Gaussian matrix $X$.  Then, one can either choose the position of a sub-block to be sketched randomly on each step, or fix a specific well-conditioned sub-block of $A$ and just resample the matrix $X$ from the Gaussian distribution. 

We can also provide theoretical guarantees for the convergence of the sGSM method.
\begin{theorem}\label{sparse_gauss_motzkin}
Suppose $A$ is a $m \times n$ matrix with full column rank ($m \ge n$) and let $x_*$ be a solution of the system $Ax = b$. For any initial estimate $x_0$, the sGSM method produces a sequence $\{x_k, k \ge 0\}$ of iterates that satisfy
\begin{align}\label{c_rate_sGSM}
\E&\|x_{k} - x_*\|_2^2 \le \left( 1 - c \frac{\log s}{\tilde\kappa(A')}\right)^k \|x_0- x_*\|_2^2.
\end{align}
Here, $c > 0$ is an absolute constant and $A'$ has maximal $\tilde\kappa(A')$ over all  $s \times n$ row submatrices of $A$.
\end{theorem}
\begin{remark} If we know the condition number $\tilde\kappa'$ of a $s \times n$ row block that we use in iterations, Theorem~\ref{c_rate_sGSM} holds with $\tilde\kappa'$ in place of the worst case $\tilde\kappa(A')$. There is an abundant literature on existence and construction of well-conditioned sub-blocks for any matrix (see~\cite{tropp2009column, vershynin2006random, tzafriri1991problem} as well as the discussion in \cite{NeeTro}). Also, in many standard cases a random selection of a sub-block produces a well-conditioned submatrix (see~\cite{tropp2008norms, chretien2012invertibility}). For our experiments in Section~\ref{experiments} we split the matrix into $\lfloor m/s\rfloor$ blocks of the same size and every time pick a sub-block randomly with replacement.
\end{remark}
\subsection{Proof of the theoretical results}\label{proof_with_error}

\begin{proposition}\label{main_motzkin} Suppose $A$ is a $m \times n$ matrix with full column rank ($m \ge n$) and let $x_*$ be a solution of the system $Ax = b$.  Let $x_k$ be a fixed vector in $\R^n$. Let $S$ be $m \times s$ matrix with i.i.d. standard normal entries and $x_{k+1}$ is obtained by iteration \eqref{iter} with the index $i$ chosen by the rule~\eqref{sketched condition}. Then, 
\begin{align*}
\E_S\|x_{k+1} - &x_*\|_2^2 \le (1 - c\frac{\log s \cdot \sigma^2_{min}(A)}{\|A\|_F^2})\|x_k - x_*\|_2^2 .
\end{align*}
\end{proposition}
\begin{proof}
Using that $x_{k+1} - x_*$ is orthogonal to $(S^TA)_i$, where $i$ is the index chosen for the $k$-th iteration, we have the following by Pythagorean theorem:
\begin{align*}
\|x_{k+1} - x_*\|_2^2 = \|x_k - x\|_2^2 - \frac{\|(S^Tb)_i - (S^TA)_i x_k\|_2^2}{\|(S^TA)_i\|_2^2} \\
=  \|x_k - x\|_2^2 - \frac{\|S^Tb - S^TA x_k\|_{\infty}^2}{\|(S^TA)_i\|_2^2}.
\end{align*}
since $i = \argmax ((S^Tb)_i - (S^TA)_ix_k)^2$.

Now, since we consider the previous iteration fixed (so, $x_k$ and $i$ are fixed),
$$
\E_S \|x_{k+1} - x_*\|_2^2 =\|x_k - x\|_2^2 - \E \frac{\|S^Tb - S^TA x_k\|_{\infty}^2}{\|(S^TA)_i\|_2^2}.
$$
Using that the function $(x, y) \mapsto x^2/y$ is convex on the positive orthant, by Jensen's inequality, we can estimate
\begin{align*}
\E \frac{\|S^Tb - S^TA x_k\|_{\infty}^2}{\|(S^TA)_i\|_2^2} \ge \frac{(\E_S\|S^Tb - S^TA x_k\|_{\infty})^2}{\E (\|(S^TA)_i\|_2^2)} \\
=  \frac{(\E_S\|S^TA ( x_* - x_k)\|_{\infty})^2}{\|A\|_F^2},
\end{align*}
where in the denominator we computed
$$
\E\|(S^TA)_i\|_2^2 = \sum_{k =1}^n \E (\sum_{j = 1}^m S_{ij} A_{jk})^2 = \sum_{kj} A_{jk}^2 = \|A\|^2_F.
$$
Moreover, from the estimate for the maximum of independent normal random variables,
$$
\E_S\|S^TA ( x_* - x_k)\|_{\infty} = \E \max_{i \in [s]} \langle S_i, v \rangle \ge c \|v\|_2 \sqrt{\log s},
$$
where $v = A(x - x_k)$, then
\begin{align*}
\E_S \|x_{k+1} - x_*\|_2^2 \le \|x_k - x\|_2^2 - \frac{c \|v\|_2^2 \log{s}}{\|A\|_F^2} \\
\le \|x_k - x\|_2^2 - \frac{c \sigma_{min}^2(A) \log{s}}{\|A\|_F^2} \|x_k - x\|_2^2.
\end{align*}
This concludes the proof of the Proposition~\ref{main_motzkin}.
\end{proof}

{\bfseries Proof of Theorem~\ref{gauss_motzkin}.} We have
\begin{align*}
\E\|x_{k} - &x_*\|_2^2  = \E_{S_1} \E_{S_2}  \ldots \E_{S_k} \|x_{k} - x_*\|_2^2 \\
&\le \left[ 1 -  c\frac{\log s \cdot \sigma^2_{min}(A)}{\|A\|_F^2}\right]^k \|x_0- x_*\|_2^2.
\end{align*}
Here, $\E_{S_1}, \ldots, \E_{S_k}$ refer to the randomness of choosing a matrix $S_i$ (independent from all Gaussian sketches $S_1, \ldots, S_{i-1}$ that were used during previous steps). The last inequality is guaranteed by Proposition~\ref{main_motzkin}. This concludes the proof of Theorem~\ref{gauss_motzkin}.

{\bfseries Proof of Theorem~\ref{sparse_gauss_motzkin}.} Note that for the sparse Gaussian sketch defined by~\eqref{sparse_sample}, $S^TA = X^TA'$, where $A'$ is an $s \times n$ row sub-block of $A$.  Proposition~\ref{main_motzkin} holds for $A = A'$ and $S =X$. Then, the proof is concluded along the lines of the proof of Theorem~\ref{gauss_motzkin}.

\section{Experimental data and results}\label{experiments}
In this section, we present some numerical experiments on the performance of the Sketched Motzkin methods. Everything was coded and run in MATLAB R2019a, on a 1.6GHz dual-core Intel Core i5, 8 GB 2133 MHz. 
Time is always measured in seconds. We generate the solution of a system as a random vector (and define the left hand side $b$ accordingly), so we do not need to worry about the case when $\|x_*\|_2 = 0$. We use $x_0 = 0$ as an initial point.

The first two figures show per iteration performance on two main models of the matrix $A$: an incoherent Gaussian matrix with i.i.d. $N(0,1)$ entries and a coherent matrix with almost collinear rows, $A_{ij}\sim$ Unif$[0.8,1]$. 

As expected, Gaussian sketching is as good as block sketching (SKM) when the rows of $A$ are Gaussian themselves, so SKM, GSM and sGSM show the same performance on the Gaussian model (Figure~\ref{iter_gaus}). Also not surprisingly, choosing the most violated equation (Motzkin) shows the best per iteration convergence progress, whereas taking the next equation at random (Kaczmarz) is the weakest. 

\begin{figure}
\includegraphics[width=0.9\columnwidth]{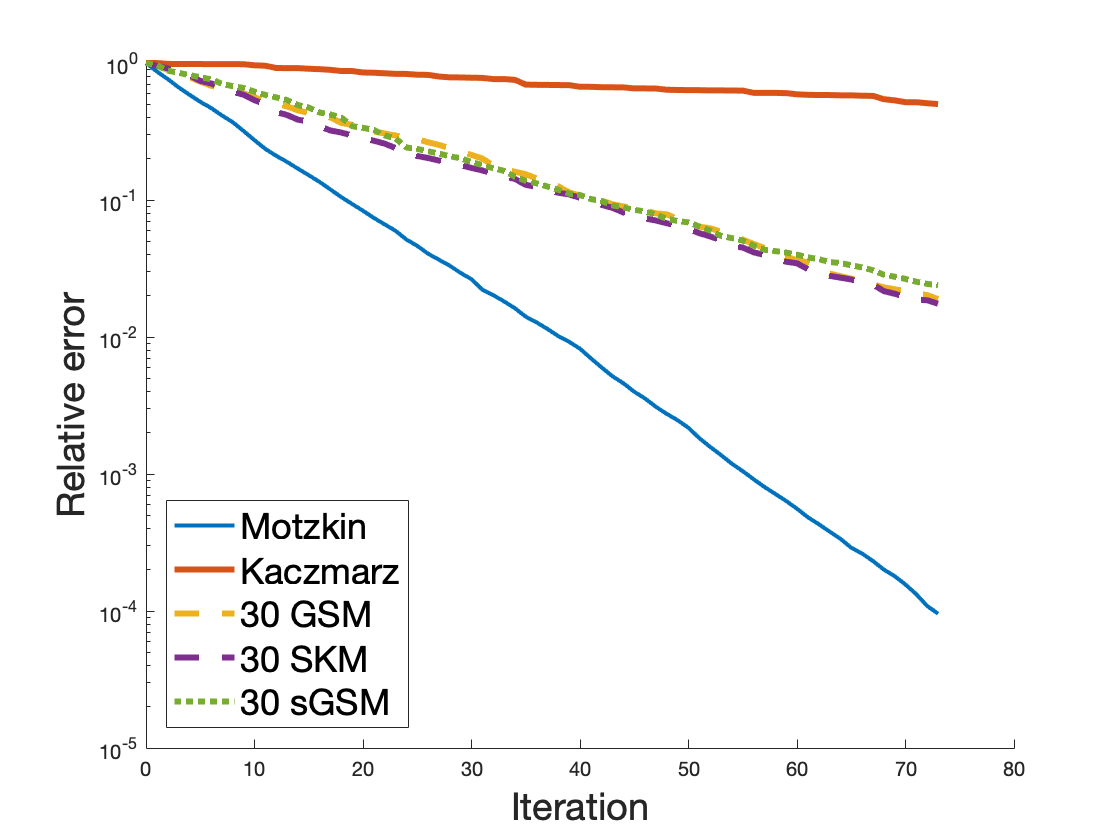} 
\caption{Performance of the methods on $A = 5000 \times 100$ matrix with  i.i.d. $N(0,1)$ entries}
\label{iter_gaus}
\end{figure}

We can see in Figure~\ref{iter_unif} that Gaussian sketching significantly improves the performance of Motzkin's method on the coherent model, and that lighter sparse sketches (sGSM) work as well as the dense GSM in practice.

\begin{figure}
\includegraphics[width=0.9\columnwidth]{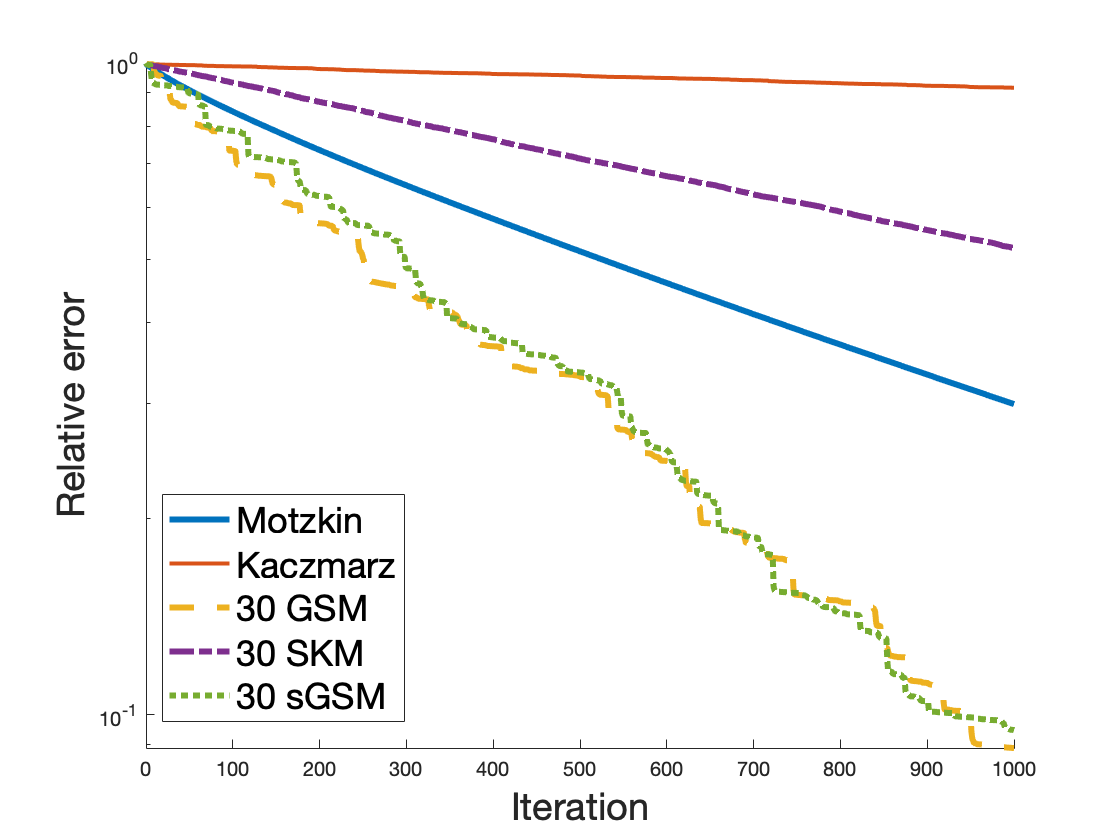}
\caption{Performance of the methods on $A = 5000 \times 100$ matrix with  i.i.d. $Unif[0.8.,1]$ entries}
\label{iter_unif}
\end{figure}

In Figure~\ref{time_unif}, we can see that this per iteration advantage is still not enough for Gaussian sketched methods to show the best performance in time. Still, it is practical to use sketching: the SKM method gives the best convergence in time. However, with the use of distributed computations or fast matrix multiplication one potentially could speed up the application of the Gaussian sketches, making GSM and sGSM methods more efficient. We can see in Figure~\ref{time_unif} the advantage of using sparse Gaussian sketches rather than dense GSM.

\begin{figure}
\includegraphics[width=0.9\columnwidth]{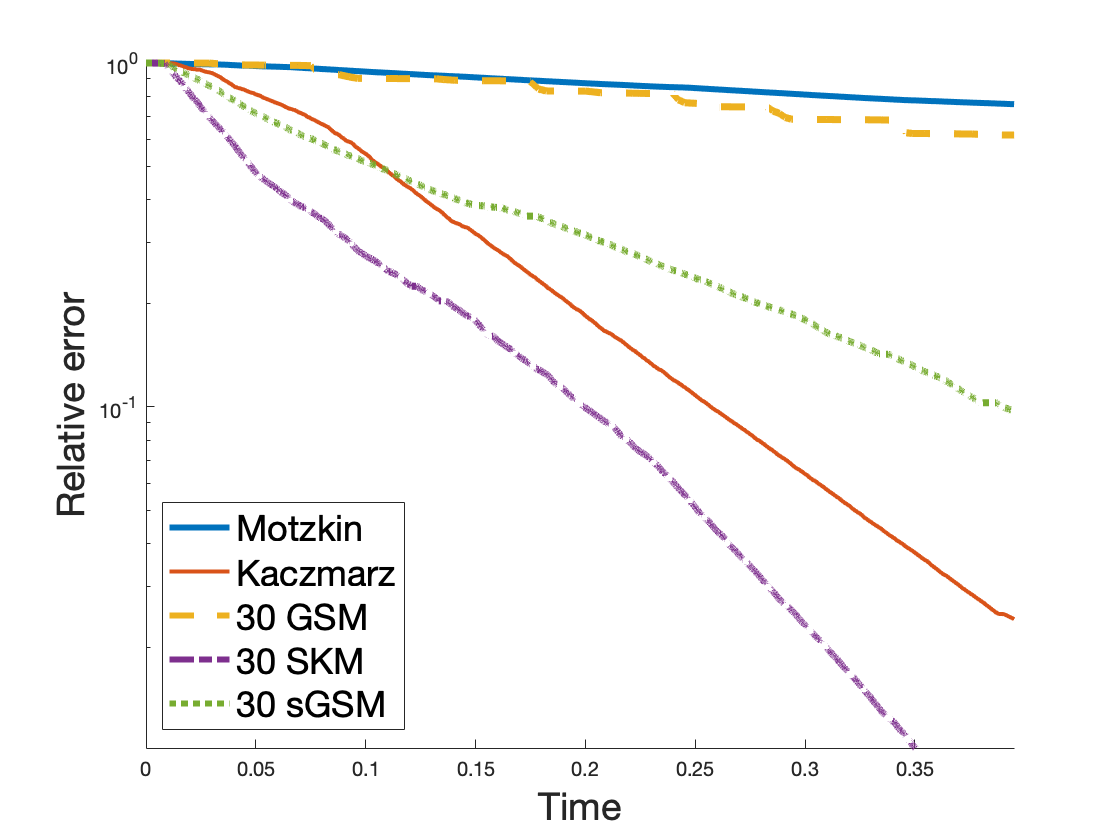} 
\caption{Performance in time on $A = 5000 \times 100$ matrix with  i.i.d. $Unif[0.8.,1]$ entries }
\label{time_unif}
\end{figure}

We also compare the preformance of the methods discussed above on two real world datasets, GAS and COVTYPE, both taken from the UCI repository \cite{Dua:2019}.

\begin{figure}
\includegraphics[width=0.9\columnwidth]{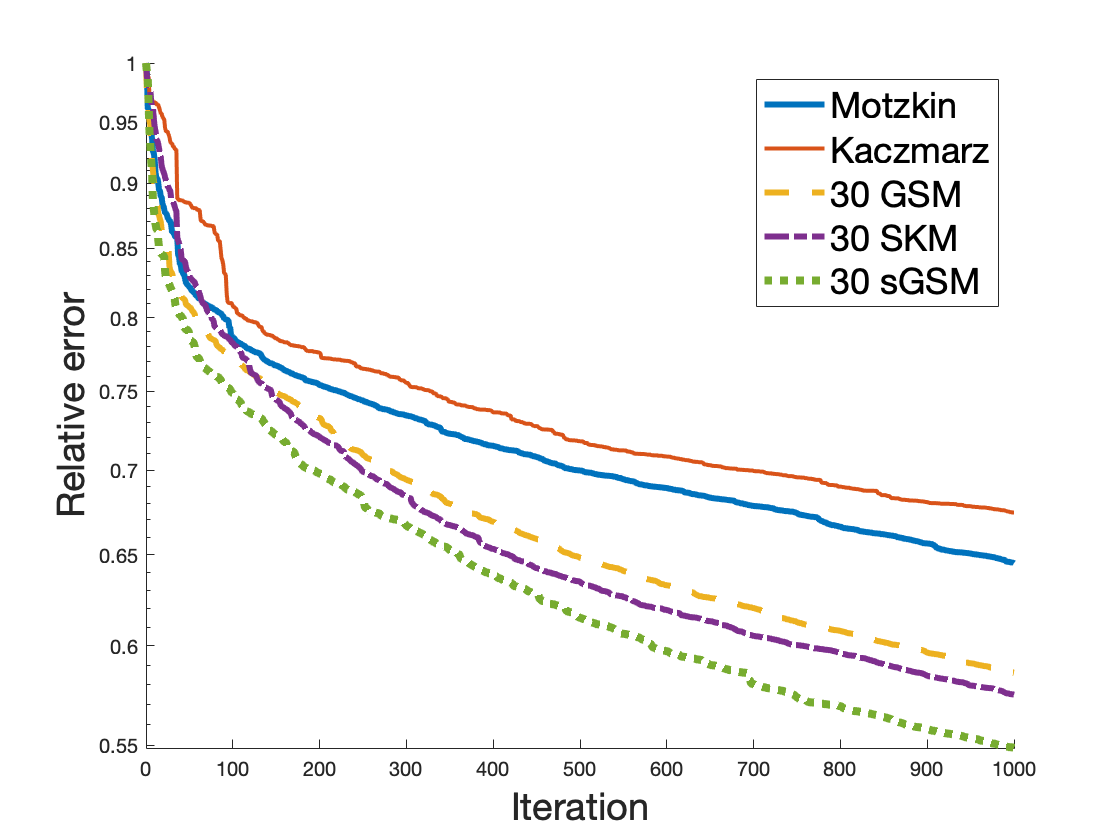} 
\caption{Performance of the methods on a real-world dataset, GAS ($1000 \times 128$)}
\label{iter_gas}
\end{figure}
\begin{figure}
\includegraphics[width=0.9\columnwidth]{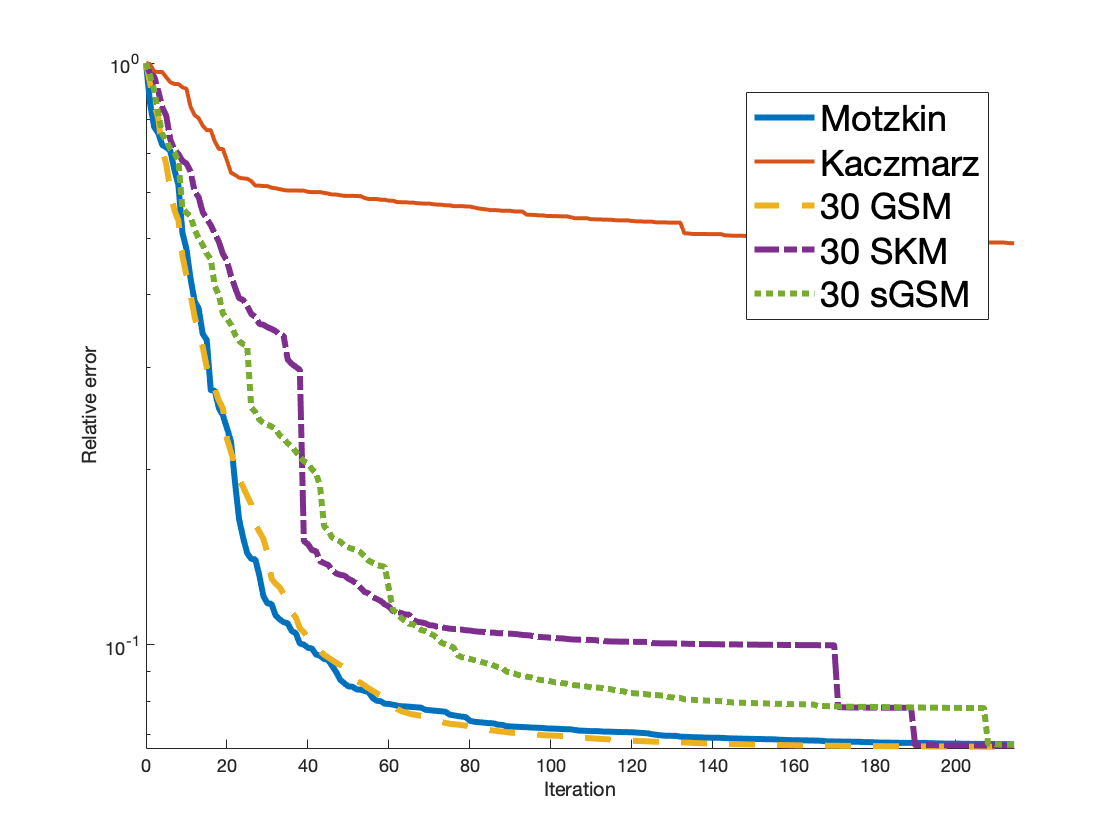}
\caption{Performance of the methods on a real-world dataset, COVTYPE ($5000 \times 54$)}
\label{iter_covtype}
\end{figure}

Finally, we compare various sketch sizes for Motzkin's method. As we can see in Figure~\ref{dep_on_block_size}, increasing the block size improves per iteration performance (which is expected since we keep more and more information about the system). Clearly, bigger blocks mean slower iterations, so the question about an optimal block size requires the comparison in time. In Figure~\ref{opt_block_size}, we compare the time required to achieve certain error threshold for sGSM methods, varying the block size $s$ from one to $n = 100$. We observe that there is a non-trivial optimal block size (somewhere between $5$ and $20$ in our experiments). 

\begin{figure}
\includegraphics[width=0.9\columnwidth]{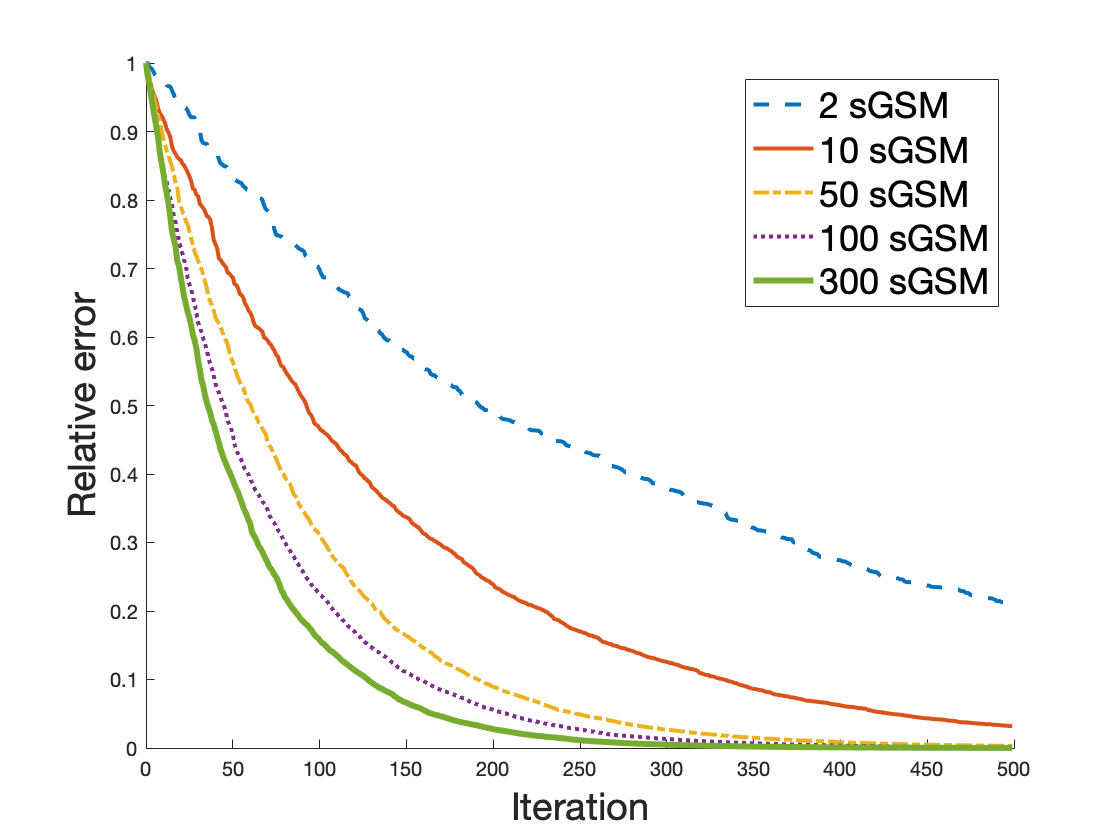}
\caption{Dependence of the convergence rate on the block size. $A = 5000 \times 100$ matrix with  i.i.d. $N(0,1)$ entries}
\label{dep_on_block_size}
\end{figure}
\begin{figure}
\includegraphics[width=0.9\columnwidth]{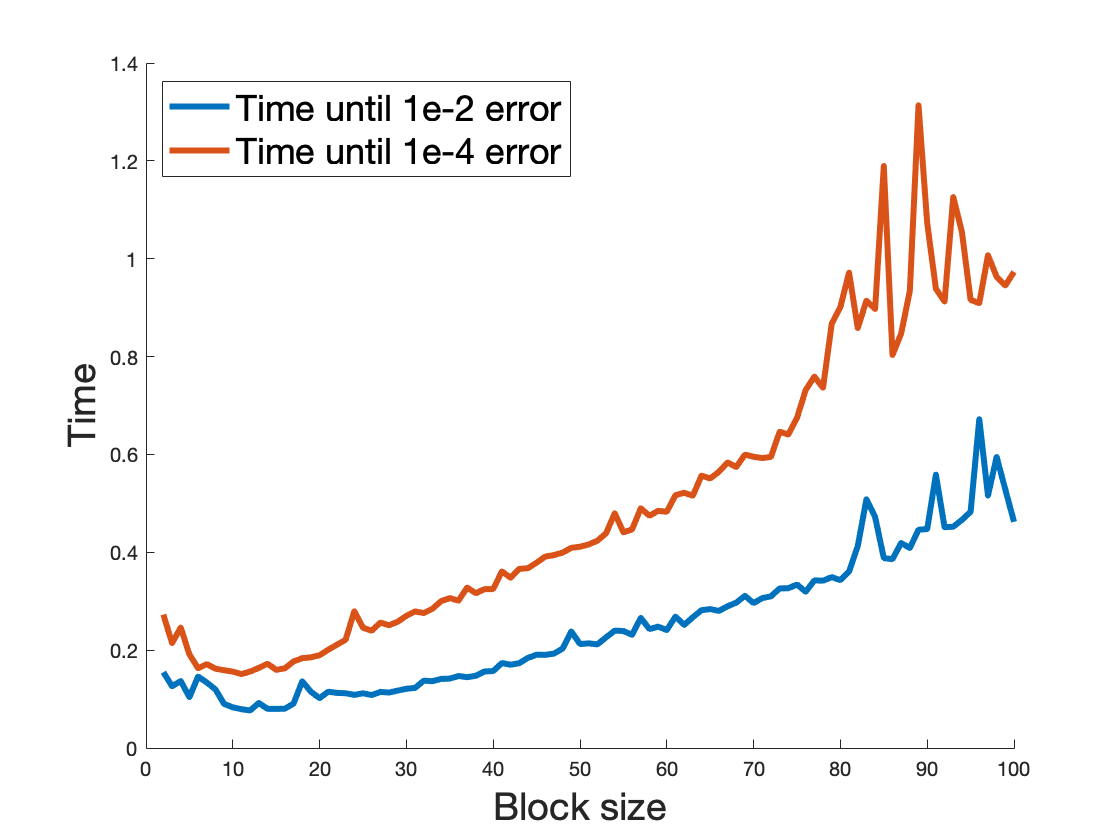}
\caption{sGSM method: $A = 5000 \times 100$ i.i.d. matrix from $Unif[0.8,1]$ model; time is averaged over 20 iterations}
\label{opt_block_size}
\end{figure}

\section{Conclusions} \label{conclusions}
In this paper, we analyzed the application of sketching to Motzkin's iterative method for solving consistent overdetermined large system of equations. We considered three ways to sketch Motzkin's method: SKM, GSM and sparseGSM. We provide theoretical guarantees for the accelerated convergence of GSM (and sparseGSM for a well-conditioned matrix). In the experiments section, we have shown some cases when sketched methods work better than both Kaczmarz and Motzkin (and when Gaussian sketches outperform SKM method). Finally, we investigate experimentally optimal block size for the sparseGSM method. One of the interesting future directions of the current work would be to investigate the dependence of the optimal block size on the model and provide theoretical estimates for it.

\section{Acknowledgement} The authors are grateful to and were partially supported by NSF CAREER DMS $\#1348721$ and NSF BIGDATA  $\#1740325$. We also acknowledge sponsorship by Capital Fund Management. 
\bibliography{liza-bib}
\end{document}